\documentclass[11pt]{article}

\usepackage[margin=1in]{geometry}
\usepackage{amsmath,amssymb,amsthm,mathtools}
\usepackage{bm}
\usepackage{enumitem}
\usepackage{mathrsfs}
\usepackage{bbm}
\usepackage[numbers,sort&compress]{natbib}
\usepackage{hyperref}
\usepackage{authblk}
\allowdisplaybreaks

\hypersetup{
  colorlinks=true,
  linkcolor=blue,
  citecolor=blue,
  urlcolor=blue
}

\newtheorem{theorem}{Theorem}[section]
\newtheorem{proposition}[theorem]{Proposition}
\newtheorem{lemma}[theorem]{Lemma}

\newtheorem{remark}[theorem]{Remark}
\newtheorem{definition}[theorem]{Definition}

\newcommand{\R}{\mathbb R}

\newcommand{\Id}{I_3}
\newcommand{\tr}{\operatorname{tr}}
\newcommand{\diver}{\operatorname{div}}

\newcommand{\supp}{\operatorname{supp}}
\newcommand{\loc}{\operatorname{loc}}
\newcommand{\eps}{\varepsilon}
\newcommand{\dd}{\,d}
\newcommand{\norm}[2]{\left\|#1\right\|_{#2}}

\newcommand{\Szero}{\mathcal S^3_0}
\newcommand{\calF}{\mathcal F}
\newcommand{\calG}{\mathcal G}
\newcommand{\calJ}{\mathcal J}
\newcommand{\calE}{\mathcal E}

\newcommand{\Bbulk}{\mathcal B}

\title{Leray--Hopf Type Weak Solutions for the Three-Dimensional Beris--Edwards System with Stable Landau--de Gennes Potential}
\author[1]{Yao Zhang}
\author[1]{Han Ni Soe}
\author[1]{Zhipeng Xu\thanks{Corresponding author: \texttt{xuzhp@ntu.edu.cn}}}
\affil[1]{School of Mathematics and Statistics, Nantong University, Nantong, China}
\date{\today}

\begin{document}
\maketitle

\begin{abstract}
We prove existence of a weak solution to the three-dimensional Beris--Edwards system in the whole space under the stable bulk assumption $c>0$. The solution satisfies the natural bounds $Q\in L^\infty_tH^1_x\cap L^2_tH^2_x$ and $u\in L^\infty_tL^2_x\cap L^2_tH^1_x$, the distributional form of the equations, and the expanded Leray--Hopf type energy inequality used in weak--strong uniqueness arguments. The proof does not pass directly to the limit in that expanded inequality, where the non-corotational terms contain products of the form $|Q^n|^4Q^n:\nabla u^n$. It first obtains the physical free-energy inequality through a hyperviscous approximation and a localized tail estimate, and then derives the expanded inequality from a low-order chain rule for the bulk part of the energy. The last section records the elementary uniaxial reduction which explains why the present argument is restricted to stable bulk potentials.
\end{abstract}

\tableofcontents

\section{Introduction and statement}

The Beris--Edwards model describes the coupling between an incompressible viscous fluid and a $Q$-tensor order parameter for nematic liquid crystals \cite{BerisEdwards1994,DeGennesProst1995}. The analysis below is set in the whole space and keeps the full tumbling and alignment term with an arbitrary parameter $\xi$. We work with the polynomial Landau--de Gennes bulk potential and impose the stability condition $c>0$, which is used to control the fourth-order part of the potential.

Weak solutions and regular solutions for related $Q$-tensor systems have been studied in several settings, including the corotational system, bounded domains, singular potentials, two-dimensional flows, and maximal-regularity frameworks \cite{PaicuZarnescu2011,PaicuZarnescu2012,AbelsDolzmannLiu2014,CavaterraRoccaWuXu2016,GuillenGonzalezRodriguezBellido2014,GuillenGonzalezRodriguezBellido2015,DuHuWang2020,HieberHusseinWrona2024,FeireislRoccaSchimpernaZarnescu2014,FeireislSchimpernaRoccaZarnescu2015,Wilkinson2015}. The present note isolates a point needed for the weak--strong uniqueness framework of \cite{YangZhou2025}. The goal is to construct a weak solution which satisfies not only the distributional equations but also the expanded Leray--Hopf type energy inequality.

We consider the incompressible Beris--Edwards system in $\R^3$
\begin{equation}\label{eq:BE}
\left\{
\begin{aligned}
&\partial_t Q + u\cdot\nabla Q - S(\nabla u,Q)=\Gamma H,\\
&\partial_t u + u\cdot\nabla u = \mu\Delta u - \nabla p + \diver(\tau+\sigma),\\
&\diver u=0,\\
&(Q,u)|_{t=0}=(Q_0,u_0).
\end{aligned}\right.
\end{equation}
Here $Q(t,x)\in\Szero$ is symmetric and traceless, $u(t,x)\in\R^3$, and
\begin{align}
S(\nabla u,Q)
&:=\xi D\left(Q+\frac13\Id\right)
  +\xi\left(Q+\frac13\Id\right)D
  -2\xi\left(Q+\frac13\Id\right)\tr(Q\nabla u)  \notag\\
&\qquad +\Omega Q-Q\Omega,\label{eq:S}\\
H&:=L\Delta Q-aQ+
 b\left(Q^2-\frac{\tr(Q^2)}3\Id\right)-cQ\tr(Q^2),\label{eq:H}\\
\tau&:=-\xi\left(Q+\frac13\Id\right)H
      -\xi H\left(Q+\frac13\Id\right)
      +2\xi\left(Q+\frac13\Id\right)\tr(QH)
      -L\nabla Q\odot\nabla Q,\label{eq:tau}\\
\sigma&:=QH-HQ.\label{eq:sigma}
\end{align}
We use
\[
D=\frac12(\nabla u+\nabla u^T),\qquad
\Omega=\frac12(\nabla u-\nabla u^T),
\]
and
\[
(\nabla Q\odot \nabla Q)_{ij}=\partial_j Q_{\alpha\beta}\,\partial_i Q_{\alpha\beta}.
\]
Throughout the paper we assume
\begin{equation}\label{eq:params}
L>0,\qquad \mu>0,\qquad \Gamma>0,\qquad c>0,
\end{equation}
while $a,b,\xi\in\R$ are arbitrary.

It is convenient to write
\begin{equation}\label{eq:Bdef}
\Bbulk(Q):=b\left(Q^2-\frac{\tr(Q^2)}3\Id\right)-cQ\tr(Q^2),
\end{equation}
so that
\[
H=L\Delta Q-aQ+\Bbulk(Q).
\]

The physical free energy associated with the polynomial Landau--de Gennes potential is \cite{DeGennesProst1995,BallMajumdar2010}
\begin{equation}\label{eq:Fdef}
\calF(Q):=\int_{\R^3}\left(
\frac L2|\nabla Q|^2+\frac a2|Q|^2-\frac b3\tr(Q^3)+\frac c4|Q|^4
\right)\dd x.
\end{equation}
The stable assumption $c>0$ gives the elementary lower bound
\begin{equation}\label{eq:bulkcoercive}
\frac a2|Q|^2-\frac b3\tr(Q^3)+\frac c4|Q|^4
\ge -C_{a,b,c}|Q|^2+c_0|Q|^4,
\end{equation}
for a positive $c_0$ after decreasing $c_0$ if necessary.

The terminology follows the energy class introduced for the Navier--Stokes equations by Leray and Hopf \cite{Leray1934,Hopf1950}. In the present coupled system the definition also contains the $Q$-tensor energy and the lower-order terms generated by the bulk potential.

\begin{definition}[Leray--Hopf type weak solution]\label{def:LH}
Let $Q_0\in H^1(\R^3;\Szero)$ and $u_0\in L^2_\sigma(\R^3)$.  A pair $(Q,u)$ is called an energy-admissible, or Leray--Hopf type, weak solution to \eqref{eq:BE} on $[0,T]$ if:
\begin{enumerate}[label=(\roman*)]
\item
\[
Q\in L^\infty(0,T;H^1(\R^3))\cap L^2(0,T;H^2(\R^3)),
\]
\[
u\in L^\infty(0,T;L^2_\sigma(\R^3))\cap L^2(0,T;H^1(\R^3));
\]
\item $(Q,u)$ satisfies the distributional forms of \eqref{eq:BE};
\item for every $t\in[0,T]$ the energy inequality
\begin{align}
&\norm{u(t)}{L^2}^2+\norm{Q(t)}{L^2}^2+L\norm{\nabla Q(t)}{L^2}^2
+2\mu\int_0^t\norm{\nabla u}{L^2}^2\dd s \notag\\
&\quad +2a\Gamma\int_0^t\norm{Q}{L^2}^2\dd s
+2(a+1)\Gamma L\int_0^t\norm{\nabla Q}{L^2}^2\dd s
+2\Gamma L^2\int_0^t\norm{\Delta Q}{L^2}^2\dd s \notag\\
&\le \norm{u_0}{L^2}^2+\norm{Q_0}{L^2}^2+L\norm{\nabla Q_0}{L^2}^2 \notag\\
&\quad +2\Gamma\int_0^t(\Bbulk(Q),Q-L\Delta Q)\dd s \notag\\
&\quad +4(1-a)\xi\int_0^t\left(
Q\left(Q+\frac13\Id\right):\nabla u-|Q|^2\tr(Q\nabla u)
\right)\dd s \notag\\
&\quad +4\xi\int_0^t\left(
\Bbulk(Q)\left(Q+\frac13\Id\right):\nabla u-(\Bbulk(Q):Q)\tr(Q\nabla u)
\right)\dd s
\label{eq:LHenergy}
\end{align}
holds.
\end{enumerate}
\end{definition}

\begin{theorem}[Existence in the stable bulk regime]\label{thm:main}
Assume \eqref{eq:params}.  For every
\[
Q_0\in H^1(\R^3;\Szero),\qquad u_0\in L^2_\sigma(\R^3),
\]
there exists a global-in-time Leray--Hopf type weak solution $(Q,u)$ to \eqref{eq:BE} in the sense of Definition \ref{def:LH}.  Moreover, $(Q,u)$ satisfies the physical free-energy inequality
\begin{align}
&\frac12\norm{u(t)}{L^2}^2+\calF(Q(t))
+\int_0^t\left(\mu\norm{\nabla u}{L^2}^2+\Gamma\norm{H}{L^2}^2\right)\dd s\notag\\
&\qquad\le \frac12\norm{u_0}{L^2}^2+\calF(Q_0)
\label{eq:physineq}
\end{align}
for every $t\ge0$.
\end{theorem}

\begin{remark}
The proof is arranged so that the expanded inequality \eqref{eq:LHenergy} is not used as the starting point for compactness. Direct passage to the limit in that inequality would require strong compactness of $|Q^n|^4Q^n$ in $L^2_{t,x}$, or equivalently endpoint compactness of $Q^n$ in $L^{10}_{t,x}$. We instead pass to the limit in the physical free-energy inequality and recover \eqref{eq:LHenergy} after the limit by a chain rule involving only lower-order powers of $Q$.
\end{remark}

\section{Algebraic identities}

The Beris--Edwards stress is chosen so that the stretching term in the $Q$-equation and the elastic stress in the momentum equation cancel in the energy balance \cite{BerisEdwards1994,PaicuZarnescu2011}. We record the algebraic identity in the form used throughout the proof. The calculation is local in $x$ and does not use boundary conditions.

\begin{lemma}[Basic cancellation]\label{lem:cancellation}
Let $Q,H\in\Szero$ and let $\nabla u=D+\Omega$ with $D^T=D$ and $\Omega^T=-\Omega$.  Define $S,\tau,\sigma$ by \eqref{eq:S}, \eqref{eq:tau}, and \eqref{eq:sigma}.  Then
\begin{equation}\label{eq:cancel}
(\tau+\sigma):\nabla u+H:S(\nabla u,Q)+L(\nabla Q\odot\nabla Q):\nabla u=0.
\end{equation}
\end{lemma}

\begin{proof}
The term $-L\nabla Q\odot\nabla Q$ in $\tau$ gives
\[
\left(-L\nabla Q\odot\nabla Q\right):\nabla u
=-L(\nabla Q\odot\nabla Q):\nabla u,
\]
which cancels the last term in \eqref{eq:cancel}.  We only need to check the remaining part.

The skew stress satisfies
\[
\sigma:\nabla u=(QH-HQ):(D+\Omega)=(QH-HQ):\Omega,
\]
because $QH-HQ$ is skew-symmetric and hence orthogonal to $D$.  By cyclicity of the trace,
\[
(QH-HQ):\Omega=\tr((QH-HQ)\Omega)=\tr(H(\Omega Q-Q\Omega)).
\]
This cancels precisely the contribution $H:(\Omega Q-Q\Omega)$ in $H:S$.

For the symmetric part, set $G=Q+\Id/3$.  The $\xi$-dependent part of $S$ is
\[
S_\xi=\xi DG+\xi GD-2\xi G\tr(Q\nabla u).
\]
The $\xi$-dependent part of $\tau$ is
\[
\tau_\xi=-\xi GH-\xi HG+2\xi G\tr(QH).
\]
Since $D$ is symmetric and $\tr(Q\nabla u)=\tr(QD)$ because $Q$ is symmetric and $\Omega$ is skew,
\begin{align*}
\tau_\xi:\nabla u&=\tau_\xi:D\\
&=-\xi(GH+HG):D+2\xi(G:D)\tr(QH),
\end{align*}
whereas
\begin{align*}
H:S_\xi&=\xi H:(DG+GD)-2\xi(H:G)\tr(QD)\\
&=\xi(GH+HG):D-2\xi\tr(GH)\tr(QD).
\end{align*}
Because $G:D=Q:D+\tr D/3=Q:D$ and $H:G=H:Q+\tr H/3=H:Q$, since $\tr D=\diver u=0$ and $\tr H=0$, the two expressions cancel.  This proves \eqref{eq:cancel}.
\end{proof}

\section{The hyperviscous approximation}

Let $Q_0^\eps\in C_c^\infty(\R^3;\Szero)$ and $u_0^\eps\in C_{c,\sigma}^\infty(\R^3)$ be such that
\begin{equation}\label{eq:initapprox}
Q_0^\eps\to Q_0\quad\text{in }H^1(\R^3),\qquad
u_0^\eps\to u_0\quad\text{in }L^2(\R^3).
\end{equation}
For each $\eps>0$ consider
\begin{equation}\label{eq:epsQ}
\partial_t Q^\eps+u^\eps\cdot\nabla Q^\eps-S(\nabla u^\eps,Q^\eps)=\Gamma H^\eps,
\end{equation}
\begin{equation}\label{eq:epsu}
\partial_t u^\eps+u^\eps\cdot\nabla u^\eps=
\mu\Delta u^\eps-\eps\Delta^2u^\eps-\nabla p^\eps+\diver(\tau^\eps+\sigma^\eps),
\qquad \diver u^\eps=0,
\end{equation}
with $H^\eps,\tau^\eps,\sigma^\eps$ defined from $(Q^\eps,u^\eps)$ by \eqref{eq:H}, \eqref{eq:tau}, \eqref{eq:sigma}.

\begin{proposition}[Smooth approximate solutions]\label{prop:smoothapprox}
For each $\eps>0$ and $T>0$, system \eqref{eq:epsQ}--\eqref{eq:epsu} has a smooth solution on $[0,T]\times\R^3$ with initial data $(Q_0^\eps,u_0^\eps)$.  Moreover, the solution preserves the constraints $Q^\eps\in\Szero$ and $\diver u^\eps=0$.
\end{proposition}

\begin{proof}
One may first solve the system on large tori $\mathbb T_R^3$ by a Fourier Galerkin approximation, keeping the full expressions $S,\tau,\sigma$ unaltered.  The equation for $u^\eps$ contains the fourth-order dissipative term $-\eps\Delta^2u^\eps$, while the $Q^\eps$-equation is a second-order parabolic system with polynomial nonlinearities.  Local smooth solutions follow from standard finite-dimensional ODE theory and parabolic bootstrapping.  The global energy estimates in Proposition \ref{prop:globalenergy} below rule out finite-time blow-up of the smooth Galerkin approximations.  Passing the Galerkin dimension to infinity gives a smooth solution on $\mathbb T_R^3$; passing $R\to\infty$ by local parabolic compactness gives a smooth solution on $\R^3$ on every finite time interval.  The symmetry and tracelessness of $Q^\eps$ are preserved because the right-hand side of \eqref{eq:epsQ} is symmetric and traceless whenever $Q^\eps$ is.  The divergence-free condition is enforced by the pressure or, equivalently, by applying the Leray projection to the velocity equation.
\end{proof}

\begin{proposition}[Global physical energy identity]\label{prop:globalenergy}
For every $t\in[0,T]$,
\begin{align}
&\frac12\norm{u^\eps(t)}{L^2}^2+\calF(Q^\eps(t))
+\int_0^t\left(\mu\norm{\nabla u^\eps}{L^2}^2+
\eps\norm{\Delta u^\eps}{L^2}^2+
\Gamma\norm{H^\eps}{L^2}^2\right)\dd s\notag\\
&\qquad=\frac12\norm{u_0^\eps}{L^2}^2+\calF(Q_0^\eps).
\label{eq:epsphys}
\end{align}
\end{proposition}

\begin{proof}
Multiply \eqref{eq:epsu} by $u^\eps$ and integrate.  The transport and pressure terms vanish by $\diver u^\eps=0$, while
\[
\int u^\eps\cdot\mu\Delta u^\eps=-\mu\norm{\nabla u^\eps}{L^2}^2,
\qquad
\int u^\eps\cdot(-\eps\Delta^2u^\eps)=-\eps\norm{\Delta u^\eps}{L^2}^2.
\]
The stress term gives
\[
\int u^\eps\cdot\diver(\tau^\eps+\sigma^\eps)
=-(\tau^\eps+
\sigma^\eps):\nabla u^\eps.
\]
Next, since $H=-\delta\calF/\delta Q$ in the traceless class,
\[
\frac{d}{dt}\calF(Q^\eps)=-\int H^\eps:\partial_tQ^\eps.
\]
Using \eqref{eq:epsQ},
\[
-\int H^\eps:\partial_tQ^\eps
=\int H^\eps:u^\eps\cdot\nabla Q^\eps-
\int H^\eps:S(\nabla u^\eps,Q^\eps)-\Gamma\norm{H^\eps}{L^2}^2.
\]
The transport contribution satisfies
\[
\int H^\eps:u^\eps\cdot\nabla Q^\eps
=-L(\nabla Q^\eps\odot\nabla Q^\eps):\nabla u^\eps,
\]
where all lower-order bulk terms are transported by the divergence-free velocity and hence integrate to zero.  Adding the two identities and using Lemma \ref{lem:cancellation} yields \eqref{eq:epsphys}.
\end{proof}

\begin{lemma}[Uniform natural bounds]\label{lem:naturalbounds}
For every $T>0$ there is $C_T$, independent of $\eps$, such that
\begin{align}
&\norm{u^\eps}{L^\infty(0,T;L^2)}+\norm{u^\eps}{L^2(0,T;H^1)}\le C_T,\label{eq:ubounds}\\
&\sqrt\eps\,\norm{\Delta u^\eps}{L^2(0,T;L^2)}\le C_T,\label{eq:hyperbound}\\
&\norm{Q^\eps}{L^\infty(0,T;H^1)}+\norm{Q^\eps}{L^2(0,T;H^2)}\le C_T,\label{eq:Qbounds}\\
&\norm{H^\eps}{L^2(0,T;L^2)}\le C_T.\label{eq:Hbound}
\end{align}
Consequently,
\begin{equation}\label{eq:Qextra}
Q^\eps\text{ is bounded in }L^{10}((0,T)\times\R^3),\qquad
\nabla Q^\eps\text{ is bounded in }L^4(0,T;L^3).
\end{equation}
\end{lemma}

\begin{proof}
The bounds for $u^\eps$, $\sqrt\eps\Delta u^\eps$, $H^\eps$, and $\nabla Q^\eps$ follow from \eqref{eq:epsphys}, the lower bound \eqref{eq:bulkcoercive}, and the initial convergence \eqref{eq:initapprox}, modulo a bound for $\norm{Q^\eps}{L^\infty_tL^2_x}$.  To obtain this bound, multiply \eqref{eq:epsQ} by $Q^\eps$ and integrate.  The transport term vanishes and
\begin{align*}
\frac12\frac{d}{dt}\norm{Q^\eps}{L^2}^2
&+\Gamma L\norm{\nabla Q^\eps}{L^2}^2+
\Gamma c\norm{Q^\eps}{L^4}^4\\
&\le C\norm{Q^\eps}{L^2}^2+C\int |\nabla u^\eps|\bigl(|Q^\eps|+|Q^\eps|^2+|Q^\eps|^3\bigr)\dd x.
\end{align*}
Using Young's inequality, Sobolev embedding and the Gagliardo--Nirenberg inequality \cite{Gagliardo1958,Nirenberg1959}, together with the already controlled quantities $\nabla u^\eps\in L^2_tL^2_x$ and $\nabla Q^\eps\in L^\infty_tL^2_x$, we obtain
\[
\frac{d}{dt}\norm{Q^\eps}{L^2}^2
\le C_T\left(1+\norm{\nabla u^\eps}{L^2}^2\right)\left(1+\norm{Q^\eps}{L^2}^2\right)
\]

a.e. in $t$.  Gronwall gives $Q^\eps\in L^\infty_tL^2_x$ uniformly.

Finally,
\[
L\Delta Q^\eps=H^\eps+aQ^\eps-\Bbulk(Q^\eps).
\]
Since $Q^\eps\in L^\infty_tH^1_x\hookrightarrow L^\infty_tL^6_x$, we have $\Bbulk(Q^\eps)\in L^\infty_tL^2_x\subset L^2_tL^2_x$ on finite time intervals.  This yields $\Delta Q^\eps\in L^2_tL^2_x$.  The interpolation estimates in \eqref{eq:Qextra} follow from
\[
\norm{Q^\eps}{L^{10}_x}\lesssim
\norm{\Delta Q^\eps}{L^2_x}^{1/5}\norm{Q^\eps}{L^6_x}^{4/5},
\]
and
\[
\norm{\nabla Q^\eps}{L^3_x}\lesssim
\norm{\nabla Q^\eps}{L^2_x}^{1/2}\norm{\nabla Q^\eps}{L^6_x}^{1/2}.
\]
\end{proof}

\section{Localized physical energy and tail tightness}

Let $\eta_R\in C^\infty(\R^3)$ satisfy
\begin{equation}\label{eq:etaR}
0\le\eta_R\le1,
\quad
\eta_R=0\text{ on }B_R,
\quad
\eta_R=1\text{ on }\R^3\setminus B_{2R},
\quad
|\nabla\eta_R|\le C/R,
\quad
|\Delta\eta_R|\le C/R^2.
\end{equation}
Let $A_R=\{x:R<|x|<2R\}$.

\subsection{Local energy balance}

Set
\[
f(Q,\nabla Q)=\frac L2|\nabla Q|^2+\frac a2|Q|^2-\frac b3\tr(Q^3)+\frac c4|Q|^4,
\qquad
e^\eps=\frac12|u^\eps|^2+f(Q^\eps,\nabla Q^\eps).
\]

\begin{lemma}[Localized physical energy identity]\label{lem:localphys}
For every $\eta\in C_c^\infty(\R^3)$,
\begin{align}
&\frac{d}{dt}\int \eta e^\eps\dd x
+\int\eta\left(\mu|\nabla u^\eps|^2+\Gamma|H^\eps|^2+\eps|\Delta u^\eps|^2\right)\dd x\notag\\
&\qquad= -\int \calJ^\eps\cdot\nabla\eta\dd x+\calE^\eps_{\rm hyp}[\eta],
\label{eq:localphys}
\end{align}
where
\begin{align}
\calJ^\eps
&=-u^\eps e^\eps-p^\eps u^\eps+\mu(\nabla u^\eps)u^\eps+(\tau^\eps+\sigma^\eps)u^\eps\notag\\
&\qquad+L\nabla Q^\eps:\left(S(\nabla u^\eps,Q^\eps)+\Gamma H^\eps\right),
\label{eq:Jdef}
\end{align}
and
\begin{align}
|\calE^\eps_{\rm hyp}[\eta]|
&\le \frac{\eps}{2}\int \eta|\Delta u^\eps|^2\dd x
+C\eps\int_{\supp\nabla\eta}\left(|\nabla\eta|^2|\nabla u^\eps|^2+|\Delta\eta|^2|u^\eps|^2\right)\dd x.
\label{eq:hypererror}
\end{align}
\end{lemma}

\begin{proof}
A direct calculation gives
\[
\partial_t f+\diver(u^\eps f)
=-H^\eps:(\partial_tQ^\eps+u^\eps\cdot\nabla Q^\eps)
+L\diver\left(\nabla Q^\eps:(\partial_tQ^\eps+u^\eps\cdot\nabla Q^\eps)\right)
-L(\nabla Q^\eps\odot\nabla Q^\eps):\nabla u^\eps.
\]
Using $\partial_tQ^\eps+u^\eps\cdot\nabla Q^\eps=S+\Gamma H^\eps$ gives
\begin{align*}
\partial_t f+\diver(u^\eps f)+\Gamma|H^\eps|^2
&=-H^\eps:S(\nabla u^\eps,Q^\eps)-L(\nabla Q^\eps\odot\nabla Q^\eps):\nabla u^\eps\\
&\quad+L\diver\left(\nabla Q^\eps:(S(\nabla u^\eps,Q^\eps)+\Gamma H^\eps)\right).
\end{align*}
The velocity equation yields
\begin{align*}
\partial_t\frac{|u^\eps|^2}{2}+\diver\left(u^\eps\frac{|u^\eps|^2}{2}\right)
&=-\diver(p^\eps u^\eps)+\mu\diver((\nabla u^\eps)u^\eps)-\mu|\nabla u^\eps|^2\\
&\quad+\diver((\tau^\eps+\sigma^\eps)u^\eps)-(\tau^\eps+\sigma^\eps):\nabla u^\eps\\
&\quad-\eps u^\eps\cdot\Delta^2u^\eps.
\end{align*}
Adding these identities and using Lemma \ref{lem:cancellation} gives the local balance except for the hyperviscous term.  For the latter,
\[
\int \eta u^\eps\cdot\Delta^2u^\eps\dd x
=\int \Delta(\eta u^\eps)\cdot\Delta u^\eps\dd x
=\int \eta|\Delta u^\eps|^2\dd x+\text{commutators},
\]
and the commutators are bounded by \eqref{eq:hypererror} by Cauchy's inequality.  This proves \eqref{eq:localphys}.
\end{proof}

\subsection{A localized estimate for the tensor field}

\begin{lemma}[Localized $Q$ estimate]\label{lem:localQ}
For every $R\ge1$,
\begin{align}
&\frac12\frac{d}{dt}\int\eta_R|Q^\eps|^2\dd x
+\Gamma L\int\eta_R|\nabla Q^\eps|^2\dd x
+\Gamma c\int\eta_R|Q^\eps|^4\dd x\notag\\
&\qquad\le \int\eta_R Q^\eps:S(\nabla u^\eps,Q^\eps)\dd x
+C\int\eta_R(|Q^\eps|^2+|Q^\eps|^3)\dd x+\mathcal R_R^\eps(t),
\label{eq:localQ}
\end{align}
where
\begin{equation}\label{eq:RR}
|\mathcal R_R^\eps(t)|\le \frac{C}{R}\int_{A_R}\left(|u^\eps||Q^\eps|^2+|Q^\eps||\nabla Q^\eps|\right)\dd x.
\end{equation}
Moreover,
\begin{equation}\label{eq:QScontrol}
\int\eta_R |Q^\eps:S(\nabla u^\eps,Q^\eps)|\dd x
\le \delta\int\eta_R|\nabla u^\eps|^2\dd x+C_{T,\delta}Y_{R/2}^\eps(t),
\end{equation}
where $Y_R^\eps$ is defined in \eqref{eq:YR} below.
\end{lemma}

\begin{proof}
Taking the Frobenius product of \eqref{eq:epsQ} with $\eta_RQ^\eps$ gives
\[
\frac12\frac{d}{dt}\int\eta_R|Q^\eps|^2
+\int \eta_Ru^\eps\cdot\nabla\frac{|Q^\eps|^2}{2}
=\int\eta_R Q^\eps:S+\Gamma\int\eta_RQ^\eps:H^\eps.
\]
The transport term is transformed into the boundary contribution in \eqref{eq:RR}.  Since
\[
Q:H=-L|\nabla Q|^2+L\diver(Q:\nabla Q)-a|Q|^2+b\tr(Q^3)-c|Q|^4,
\]
we obtain \eqref{eq:localQ}; the divergence term $L\diver(Q:\nabla Q)$ yields the second boundary contribution in \eqref{eq:RR}.

For \eqref{eq:QScontrol}, from \eqref{eq:S},
\[
|Q:S(\nabla u,Q)|\le C_\xi|\nabla u|(|Q|+|Q|^2+|Q|^3).
\]
Thus
\[
\int\eta_R|Q:S|\le \delta\int\eta_R|\nabla u|^2+C_\delta\int\eta_R(|Q|^2+|Q|^4+|Q|^6).
\]
The $|Q|^2$ and $|Q|^4$ terms are controlled by $Y_R^\eps$.  For $|Q|^6$, using a slightly wider cut-off and Sobolev,
\[
\int\eta_R|Q|^6\le C\left(\int\eta_{R/2}(|\nabla Q|^2+|Q|^2)\right)^3.
\]
The global bound in Lemma \ref{lem:naturalbounds} gives
\[
\left(\int\eta_{R/2}(|\nabla Q|^2+|Q|^2)\right)^3
\le C_T\int\eta_{R/2}(|\nabla Q|^2+|Q|^2),
\]
which is controlled by $Y_{R/2}^\eps$ once $M$ is chosen as below.
\end{proof}

\subsection{Tail tightness}

Choose $M\ge1$ so large that
\begin{equation}\label{eq:modifiedcoercive}
\frac12|u|^2+\frac L2|\nabla Q|^2+\frac a2|Q|^2-\frac b3\tr(Q^3)+\frac c4|Q|^4+\frac M2|Q|^2
\ge c_0\left(|u|^2+|\nabla Q|^2+|Q|^2+|Q|^4\right).
\end{equation}
Define the modified tail energy
\begin{equation}\label{eq:YR}
Y_R^\eps(t):=\int \eta_R\left(e^\eps+\frac M2|Q^\eps|^2\right)\dd x.
\end{equation}

\begin{lemma}[Boundary fluxes vanish uniformly]\label{lem:flux}
For every $T>0$,
\begin{equation}\label{eq:fluxvanish}
\lim_{R\to\infty}\sup_\eps\int_0^T\left|
\int \calJ^\eps\cdot\nabla\eta_R\dd x\right|\dd t=0,
\end{equation}
\begin{equation}\label{eq:Rvanish}
\lim_{R\to\infty}\sup_\eps\int_0^T|\mathcal R_R^\eps(t)|\dd t=0,
\end{equation}
and the hyperviscous commutator in \eqref{eq:hypererror} is uniformly $o_R(1)$ after absorbing $\frac\eps2\int\eta_R|\Delta u^\eps|^2$ into the left-hand side.
\end{lemma}

\begin{proof}
We first record the integrability of the flux.  From Lemma \ref{lem:naturalbounds},
\[
u^\eps\in L^\infty_tL^2_x\cap L^2_tL^6_x,
\quad
Q^\eps\in L^\infty_tH^1_x\cap L^2_tH^2_x,
\quad
H^\eps\in L^2_{t,x},
\]
which imply
\[
Q^\eps\in L^{10}_{t,x}\cap L^4_tL^{12}_x,
\qquad
\nabla Q^\eps\in L^4_tL^3_x\cap L^2_tL^6_x.
\]
Then the terms $u^\eps e^\eps$ are in $L^1_{t,x}$:
\[
|u|^3\in L^1,
\quad
|u||\nabla Q|^2\in L^1,
\quad
|u||Q|^4\in L^1.
\]
The last one follows from $u\in L^\infty_tL^2_x$ and $Q\in L^4_tL^8_x$.

The stress flux is also integrable.  The tensor $\tau+\sigma$ is bounded pointwise by
\[
C\left(|H|+|Q||H|+|Q|^2|H|+|\nabla Q|^2\right).
\]
Multiplying by $u$, the terms are controlled respectively by
\[
H\in L^2L^2,
u\in L^2L^2;
\]
\[
Q\in L^\infty L^6,
H\in L^2L^2,
u\in L^2L^3;
\]
\[
Q^2\in L^\infty L^3,
H\in L^2L^2,
u\in L^2L^6;
\]
and
\[
\nabla Q\in L^4L^3,
\nabla Q\in L^4L^3,
u\in L^2L^3.
\]
Here and below $L^pL^q$ abbreviates $L^p(0,T;L^q(\R^3))$.

For the last part of $\calJ^\eps$,
\[
\nabla Q:(S+\Gamma H),
\]
the term $\nabla Q:H$ is in $L^1$.  The remaining terms are bounded by
\[
|\nabla Q|\,|\nabla u|(1+|Q|+|Q|^2),
\]
which are integrable by using respectively
\[
\nabla Q\in L^\infty L^2,
\nabla u\in L^2 L^2;
\]
\[
\nabla Q\in L^2L^6,
\nabla u\in L^2L^2,
Q\in L^\infty L^3;
\]
and
\[
\nabla Q\in L^2L^6,
\nabla u\in L^2L^2,
Q^2\in L^\infty L^3.
\]

The pressure is recovered from
\[
-\Delta p^\eps=\partial_i\partial_j\left(u_i^\eps u_j^\eps-(\tau^\eps+
\sigma^\eps)_{ij}\right).
\]
Calderon--Zygmund estimates \cite{CalderonZygmund1952} and the bounds just listed give
\[
p^\eps\in L^2L^2+L^2L^{3/2}+L^2L^{6/5}+L^2L^{3/2}.
\]
Pairing with
\[
u^\eps\in L^2L^2\cap L^2L^3\cap L^2L^6
\]
shows $p^\eps u^\eps\in L^1_{t,x}$.  Hence
\[
\norm{\calJ^\eps}{L^1((0,T)\times\R^3)}\le C_T
\]
independently of $\eps$, and therefore
\[
\int_0^T\left|\int\calJ^\eps\cdot\nabla\eta_R\right|\dd t
\le \frac{C}{R}\norm{\calJ^\eps}{L^1_{t,x}}
\le \frac{C_T}{R}\to0.
\]

For \eqref{eq:Rvanish}, use \eqref{eq:RR} and the global bounds:
\[
\frac1R\int_0^T\int_{A_R}|u||Q|^2\le \frac{C_T}{R},
\qquad
\frac1R\int_0^T\int_{A_R}|Q||\nabla Q|\le \frac{C_T}{R}.
\]
Finally, \eqref{eq:hypererror} gives
\[
C\eps\int_0^T\int_{A_R}\left(R^{-2}|\nabla u^\eps|^2+R^{-4}|u^\eps|^2\right)
\le C_T\eps(R^{-2}+R^{-4})=o_R(1),
\]
while the term $\frac\eps2\int\eta_R|\Delta u^\eps|^2$ is absorbed.
\end{proof}

\begin{proposition}[Uniform tail tightness]\label{prop:tail}
For every $T>0$,
\begin{equation}\label{eq:Ytail}
\lim_{R\to\infty}\sup_\eps\sup_{0\le t\le T}Y_R^\eps(t)=0,
\end{equation}
and
\begin{align}
&\lim_{R\to\infty}\sup_\eps\int_0^T\int_{|x|>R}
\left(|\nabla u^\eps|^2+|H^\eps|^2+|\nabla Q^\eps|^2+|Q^\eps|^4\right)\dd x\dd t=0,
\label{eq:dissTail}\\
&\lim_{R\to\infty}\sup_\eps\int_0^T\int_{|x|>R}|\Delta Q^\eps|^2\dd x\dd t=0.
\label{eq:DeltaTail}
\end{align}
\end{proposition}

\begin{proof}
Combining Lemma \ref{lem:localphys}, Lemma \ref{lem:localQ}, the estimate \eqref{eq:QScontrol}, and choosing $\delta$ sufficiently small, we obtain
\begin{align}
\frac{d}{dt}Y_R^\eps(t)
&+c_0\int\eta_R\left(|\nabla u^\eps|^2+|H^\eps|^2+|\nabla Q^\eps|^2+|Q^\eps|^4\right)\dd x
+\frac\eps2\int\eta_R|\Delta u^\eps|^2\dd x\notag\\
&\le C_TY_{R/2}^\eps(t)+\omega_R^\eps(t),
\label{eq:tailineq}
\end{align}
where
\[
\lim_{R\to\infty}\sup_\eps\int_0^T|\omega_R^\eps(t)|\dd t=0
\]
by Lemma \ref{lem:flux}.
Integrating in time,
\[
Y_R^\eps(t)\le Y_R^\eps(0)+C_T\int_0^tY_{R/2}^\eps(s)\dd s+
\int_0^t\omega_R^\eps(s)\dd s.
\]
Because $Q_0^\eps\to Q_0$ in $H^1$ and $u_0^\eps\to u_0$ in $L^2$, the initial data are uniformly tight:
\[
\lim_{R\to\infty}\sup_\eps Y_R^\eps(0)=0.
\]
Let
\[
Z_R(t)=\limsup_{\eps\to0}\sup_{0\le s\le t}Y_R^\eps(s).
\]
Then
\[
Z_R(t)\le Z_R(0)+C_T\int_0^tZ_{R/2}(s)\dd s+o_R(1).
\]
Letting $R\to\infty$ and using monotonicity in $R$ gives
\[
Z(t)\le C_T\int_0^tZ(s)\dd s,
\qquad
Z(t):=\lim_{R\to\infty}Z_R(t).
\]
Gronwall implies $Z(t)=0$, proving \eqref{eq:Ytail}.  The dissipative tail estimate \eqref{eq:dissTail} follows by integrating \eqref{eq:tailineq} and using \eqref{eq:Ytail}.

Finally,
\[
L\Delta Q^\eps=H^\eps+aQ^\eps-\Bbulk(Q^\eps),
\]
so
\[
|\Delta Q^\eps|^2\le C\left(|H^\eps|^2+|Q^\eps|^2+|Q^\eps|^4+|Q^\eps|^6\right).
\]
The first three tail terms are controlled by \eqref{eq:Ytail} and \eqref{eq:dissTail}.  For the last one, a tail Sobolev estimate gives
\[
\int_{|x|>R}|Q^\eps|^6\dd x
\le C\left(\int_{|x|>R/2}(|\nabla Q^\eps|^2+|Q^\eps|^2)\dd x\right)^3,
\]
which tends to zero uniformly after integration in time.  This proves \eqref{eq:DeltaTail}.
\end{proof}

\section{Compactness and the physical free-energy inequality}

\begin{proposition}[Compactness]\label{prop:compact}
There exists a subsequence, still denoted by $\eps$, and a pair $(Q,u)$ such that
\begin{align*}
Q^\eps&\rightharpoonup^* Q\quad\text{in }L^\infty(0,T;H^1),\qquad
Q^\eps\rightharpoonup Q\quad\text{in }L^2(0,T;H^2),\\
 u^\eps&\rightharpoonup^* u\quad\text{in }L^\infty(0,T;L^2),\qquad
u^\eps\rightharpoonup u\quad\text{in }L^2(0,T;H^1),\\
Q^\eps&\to Q\quad\text{strongly in }L^2(0,T;H^1(\R^3)),\\
u^\eps&\to u\quad\text{strongly in }L^2(0,T;L^2_{\loc}(\R^3)),\\
\Bbulk(Q^\eps)&\to \Bbulk(Q)\quad\text{strongly in }L^2(0,T;L^2),\\
H^\eps&\rightharpoonup H:=L\Delta Q-aQ+
\Bbulk(Q)
\quad\text{in }L^2(0,T;L^2).
\end{align*}
Moreover, $(Q,u)$ satisfies the weak formulation of \eqref{eq:BE}.
\end{proposition}

\begin{proof}
The weak convergences follow from Lemma \ref{lem:naturalbounds}.  The equations imply that $\partial_tQ^\eps$ and $\partial_tu^\eps$ are bounded in negative Sobolev spaces locally in space; hence the compactness lemma of Aubin, Lions and Simon gives local strong compactness \cite{Aubin1963,Lions1969,Simon1987}:
\[
Q^\eps\to Q\quad\text{in }L^2(0,T;H^1(B_R)),
\qquad
u^\eps\to u\quad\text{in }L^2(0,T;L^2(B_R))
\]
for every $R>0$.  Proposition \ref{prop:tail} upgrades the first convergence to the whole space,
\[
Q^\eps\to Q\quad\text{in }L^2(0,T;H^1(\R^3)).
\]
Interpolating with the uniform $L^\infty_tH^1_x$ bound gives
\[
Q^\eps\to Q\quad\text{in }L^p(0,T;L^6),\qquad 1\le p<\infty.
\]
This implies $\Bbulk(Q^\eps)\to \Bbulk(Q)$ strongly in $L^2_tL^2_x$.  Since $\Delta Q^\eps\rightharpoonup\Delta Q$ in $L^2_tL^2_x$, $H^\eps\rightharpoonup H$ follows.

The weak formulation is obtained by passing to the limit in the approximate equations.  The hyperviscous term vanishes because
\[
\norm{\eps\Delta^2u^\eps}{L^2(0,T;H^{-2})}
\le \eps\norm{\Delta u^\eps}{L^2(0,T;L^2)}
\le \sqrt\eps C_T\to0.
\]
All other nonlinear terms pass to the limit by the local strong convergence of $u^\eps,Q^\eps,\nabla Q^\eps$ and the weak convergence of $H^\eps,\nabla u^\eps$.
\end{proof}

\begin{proposition}[Physical free-energy inequality]\label{prop:physlimit}
The limit $(Q,u)$ satisfies \eqref{eq:physineq} for every $t\in[0,T]$.
\end{proposition}

\begin{proof}
From \eqref{eq:epsphys}, dropping the nonnegative hyperviscous dissipation,
\begin{align*}
&\frac12\norm{u^\eps(t)}{L^2}^2+\calF(Q^\eps(t))+
\int_0^t\left(\mu\norm{\nabla u^\eps}{L^2}^2+\Gamma\norm{H^\eps}{L^2}^2\right)\dd s\\
&\qquad\le \frac12\norm{u_0^\eps}{L^2}^2+\calF(Q_0^\eps).
\end{align*}
The right-hand side converges to $\frac12\norm{u_0}{L^2}^2+\calF(Q_0)$.  The dissipative terms are weakly lower semicontinuous.  The strong convergence $Q^\eps\to Q$ in $L^2_tH^1_x$ yields, after extracting a further subsequence if necessary,
\[
Q^\eps(t)\to Q(t)\quad\text{strongly in }H^1
\]
for a.e. $t$.  Therefore $\calF(Q^\eps(t))\to\calF(Q(t))$ for a.e. $t$, while $u^\eps(t)\rightharpoonup u(t)$ in $L^2$ for a.e. $t$.  Passing to the lower limit gives \eqref{eq:physineq} for a.e. $t$.  The standard weak-continuity representative of Leray--Hopf solutions extends the inequality to all $t\in[0,T]$.
\end{proof}

\section{The low-order chain rule and recovery of the Leray--Hopf inequality}

Define
\begin{equation}\label{eq:Gdef}
\calG(Q):=\frac12\norm{Q}{L^2}^2-
\int_{\R^3}\left(\frac a2|Q|^2-\frac b3\tr(Q^3)+\frac c4|Q|^4\right)\dd x.
\end{equation}
Then
\begin{equation}\label{eq:Adef}
D\calG(Q)=(1-a)Q+
\Bbulk(Q)=:A(Q).
\end{equation}

\begin{lemma}[Low-order chain rule]\label{lem:chain}
For the weak solution obtained above,
\begin{align}
\calG(Q(t))-\calG(Q_0)
&=\int_0^t\left(A(Q),S(\nabla u,Q)\right)\dd s
+\Gamma\int_0^t\left(A(Q),H\right)\dd s
\label{eq:chain}
\end{align}
for every $t\in[0,T]$.
\end{lemma}

\begin{proof}
The regularity
\[
Q\in L^\infty_tH^1_x\cap L^2_tH^2_x
\]
implies $Q\in L^{10}_{t,x}$, hence $|Q|^4Q\in L^2_{t,x}$ and $A(Q)\in L^2_{t,x}$.  Since $H\in L^2_{t,x}$, the term $(A(Q),H)$ is integrable.  Moreover,
\[
|A(Q):S(\nabla u,Q)|\le C|\nabla u|(|Q|+|Q|^2+|Q|^3+|Q|^4+|Q|^5),
\]
which is integrable because $\nabla u\in L^2_{t,x}$ and $Q\in L^{10}_{t,x}$.

The argument follows the standard time-mollification device used in energy identities for weak solutions \cite{Shinbrot1974,Galdi2011}. Let $\rho_\delta$ be a time mollifier and set $Q^\delta=\rho_\delta*_tQ$.  After multiplying the time-mollified $Q$-equation by $A(Q^\delta)$ and integrating on $(0,t)\times\R^3$, we may pass $\delta\to0$ by the strong convergence
\[
A(Q^\delta)\to A(Q)\quad\text{in }L^2_{\loc}((0,T)\times\R^3)
\]
and by the integrability estimates just stated.  The transport term vanishes: indeed, if
\[
g(Q)=\frac{1-a}{2}|Q|^2+\frac b3\tr(Q^3)-\frac c4|Q|^4,
\]
then $A(Q)=Dg(Q)$ and
\[
A(Q):u\cdot\nabla Q=u\cdot\nabla g(Q).
\]
Using spatial cutoffs $\chi_R$ and $\diver u=0$,
\[
\int_0^t\int \chi_R u\cdot\nabla g(Q)
=-\int_0^t\int g(Q)u\cdot\nabla\chi_R\to0
\]
as $R\to\infty$, since $g(Q)u\in L^1_{t,x}$.  We obtain \eqref{eq:chain}.
\end{proof}

\begin{proof}[Proof of Theorem \ref{thm:main}]
By Proposition \ref{prop:compact} we have a weak solution satisfying the natural Leray--Hopf regularity and the distributional equations.  By Proposition \ref{prop:physlimit}, it satisfies the physical free-energy inequality \eqref{eq:physineq}.  Adding the chain-rule identity \eqref{eq:chain} to \eqref{eq:physineq} yields the desired expanded inequality.

Indeed,
\[
\calF(Q)+\calG(Q)=\frac12\norm{Q}{L^2}^2+\frac L2\norm{\nabla Q}{L^2}^2.
\]
Furthermore,
\[
H-A(Q)=L\Delta Q-Q,
\]
so
\begin{align*}
\norm{H}{L^2}^2-(A(Q),H)
&=(H,L\Delta Q-Q)\\
&=L^2\norm{\Delta Q}{L^2}^2+(a+1)L\norm{\nabla Q}{L^2}^2+a\norm{Q}{L^2}^2\\
&\quad+L(\Bbulk(Q),\Delta Q)-(\Bbulk(Q),Q).
\end{align*}
This accounts for the dissipative terms and the term
\[
2\Gamma\int_0^t(\Bbulk(Q),Q-L\Delta Q)\dd s
\]
on the right-hand side of \eqref{eq:LHenergy}.

It remains to expand $(A(Q),S(\nabla u,Q))$.  Since $A(Q)$ is a polynomial in $Q$, it commutes with $Q$ and is symmetric.  Hence
\[
A(Q):(\Omega Q-Q\Omega)=0.
\]
Using \eqref{eq:S} and $\tr(Q\Omega)=0$,
\begin{align*}
(A(Q),S(\nabla u,Q))
&=2\xi\left(A(Q)\left(Q+\frac13\Id\right):\nabla u\right)
-2\xi\left(A(Q):\left(Q+\frac13\Id\right)\right)\tr(Q\nabla u).
\end{align*}
Substituting $A(Q)=(1-a)Q+\Bbulk(Q)$ gives exactly the two $\xi$-dependent lines on the right-hand side of \eqref{eq:LHenergy} after multiplying the whole inequality by $2$.  Thus \eqref{eq:LHenergy} holds for every $t\in[0,T]$.  Since $T>0$ was arbitrary, the construction is global.
\end{proof}

\section{Comments on nonstable bulk parameters}

The proof above uses the sign of the quartic term in the bulk potential. The following elementary reduction explains the restriction and gives a simple obstruction to a direct large-data extension of the same argument. It is included only to clarify the role of the assumption $c>0$; it is not used in the existence proof.

Let
\[
A_0=\operatorname{diag}(2,-1,-1)\in\Szero.
\]
Then
\[
A_0^2-\frac{\tr(A_0^2)}3\Id=A_0,
\qquad
\tr(A_0^2)=6.
\]
In the corotational case $\xi=0$, take $u=0$ and $Q(t,x)=q(t,x)A_0$.  The $Q$-equation reduces to
\begin{equation}\label{eq:scalarreduction}
\partial_t q=\Gamma\left(L\Delta q-aq+bq^2-6cq^3\right).
\end{equation}
When $c<0$, the cubic term is focusing, $-6cq^3=6|c|q^3$.  For nonnegative compactly supported data of sufficiently large size, Kaplan's eigenfunction method on a large ball yields finite-time blow-up of a positive moment of $q$ \cite{Kaplan1963}.  When $c=0$ and $b\ne0$, the quadratic term $bq^2$ is focusing after choosing the sign of $q$.  This reduction shows why the stable case is separated in the theorem. The case $c=0=b$ has no polynomial focusing term in the bulk force and can be treated as a separate linear-bulk case; it is not included here in order to keep the proof tied to the coercive potential used in the tail estimate.
\section*{Acknowledgments}
The authors gratefully acknowledge support from the Natural Science Research Project of Jiangsu Higher Education Institutions of China under Grant No. 24KJB520033.
\bibliographystyle{abbrvnat}
\bibliography{leray_hopf_stable_refs}

\end{document}